\documentclass[12pt,a4paper]{article}
\usepackage{mathrsfs}
\usepackage{amsmath}
\usepackage{amsfonts}
\usepackage{amssymb}
\usepackage{graphicx}
\usepackage{enumerate}
\usepackage{arydshln}

\begin{document}

\renewcommand{\thefootnote}{\fnsymbol{footnote}}

\newtheorem{theorem}{Theorem}[section]
\newtheorem{corollary}[theorem]{Corollary}
\newtheorem{definition}[theorem]{Definition}
\newtheorem{conjecture}[theorem]{Conjecture}
\newtheorem{question}[theorem]{Question}
\newtheorem{lemma}[theorem]{Lemma}
\newtheorem{proposition}[theorem]{Proposition}
\newtheorem{example}[theorem]{Example}
\newenvironment{proof}{\noindent {\bf
Proof.}}{\rule{3mm}{3mm}\par\medskip}
\newcommand{\remark}{\medskip\par\noindent {\bf Remark.~~}}
\newcommand{\pp}{{\it p.}}
\newcommand{\de}{\em}

\newcommand{\JEC}{{\it Europ. J. Combinatorics},  }
\newcommand{\JCTB}{{\it J. Combin. Theory Ser. B.}, }
\newcommand{\JCT}{{\it J. Combin. Theory}, }
\newcommand{\JGT}{{\it J. Graph Theory}, }
\newcommand{\ComHung}{{\it Combinatorica}, }
\newcommand{\DM}{{\it Discrete Math.}, }
\newcommand{\ARS}{{\it Ars Combin.}, }
\newcommand{\SIAMDM}{{\it SIAM J. Discrete Math.}, }
\newcommand{\SIAMADM}{{\it SIAM J. Algebraic Discrete Methods}, }
\newcommand{\SIAMC}{{\it SIAM J. Comput.}, }
\newcommand{\ConAMS}{{\it Contemp. Math. AMS}, }
\newcommand{\TransAMS}{{\it Trans. Amer. Math. Soc.}, }
\newcommand{\AnDM}{{\it Ann. Discrete Math.}, }
\newcommand{\NBS}{{\it J. Res. Nat. Bur. Standards} {\rm B}, }
\newcommand{\ConNum}{{\it Congr. Numer.}, }
\newcommand{\CJM}{{\it Canad. J. Math.}, }
\newcommand{\JLMS}{{\it J. London Math. Soc.}, }
\newcommand{\PLMS}{{\it Proc. London Math. Soc.}, }
\newcommand{\PAMS}{{\it Proc. Amer. Math. Soc.}, }
\newcommand{\JCMCC}{{\it J. Combin. Math. Combin. Comput.}, }
\newcommand{\GC}{{\it Graphs Combin.}, }
\newcommand{\LAA}{{\it Linear Algeb. Appli.}, }

\title{ \bf The skew-rank of oriented graphs\thanks{
Supported by NSFC No.11301302 and 11371205, China Postdoctoral
Science Foundation No.2013M530869, and the Natural Science
Foundation of Shandong No.BS2013SF009.  $\dag$ Corresponding
author.}}

\author{\small Xueliang Li $^a$, Guihai Yu $^{a,b, \dag}$  \\
{\small $^a$Center for Combinatorics and LPMC-TJKLC}\\
{\small Nankai University, Tianjin 300071, China.}\\
{\small $^b$Department of Mathematics} \\
{\small Shandong Institute of Business and Technology}\\
{\small Yantai, Shandong 264005, China.}\\
{\small E-mail: {\tt lxl@nankai.edu.cn; yuguihai@126.com}}\\
}

\maketitle
\vspace{-0.5cm}

\begin{abstract}
An oriented graph $G^\sigma$ is a digraph without loops and multiple
arcs, where $G$ is called the underlying graph of $G^\sigma$. Let
$S(G^\sigma)$ denote the skew-adjacency matrix of $G^\sigma$. The
rank of the skew-adjacency matrix of $G^\sigma$ is called the {\it
skew-rank} of $G^\sigma$, denoted by $sr(G^\sigma)$. The
skew-adjacency matrix of an oriented graph is skew symmetric and the
skew-rank is even. In this paper we consider the skew-rank of simple
oriented graphs. Firstly we give some preliminary results about the
skew-rank. Secondly we characterize the oriented graphs with
skew-rank 2 and characterize the oriented graphs with pendant
vertices which attain the skew-rank 4. As a consequence, we list the
oriented unicyclic graphs, the oriented bicyclic graphs with pendant
vertices which attain the skew-rank 4. Moreover, we determine the
skew-rank of oriented unicyclic graphs of order $n$ with girth $k$
in terms of matching number. We investigate the minimum value of the
skew-rank among oriented unicyclic graphs of order $n$ with girth
$k$ and characterize oriented unicyclic graphs attaining the minimum
value. In addition, we consider oriented unicyclic graphs whose
skew-adjacency matrices are nonsingular.
\end{abstract}

 \vskip 0.1cm

\section{Introduction}
\vskip 0.1cm

 \ \quad Let $G$ be a simple graph of order $n$ with vertex
set $V(G)=\{v_1,v_2,\cdots,v_n\}$ and edge set $E(G)$. The
\emph{adjacency matrix} $A(G)$ of a graph $G$ of order $n$ is the
$n\times n$ symmetric 0-1 matrix $(a_{ij})_{n\times n}$ such that
$a_{ij}=1$ if $v_i$ and $v_j$ are adjacent and 0, otherwise. We
denote by $Sp(G)$ the spectrum of $A(G)$. The rank of $A(G)$ is
called to be the rank of $G$, denoted by $r(G)$. Let $G^\sigma$ ba a
graph with an orientation which assigns to each edge of $G$ a
direction so that $G^\sigma$ becomes an oriented graph. The graph
$G$ is called the \emph{underlying graph} of $G^\sigma$. The
\emph{skew-adjacency matrix} associated to the oriented graph
$G^\sigma$ is defined as the $n\times n$ matrix
$S(G^\sigma)=(s_{ij})$ such that $s_{ij}=1$ if there has an arc from
$v_i$ to $v_j$, $s_{ij}=-1$ if there has an arc from $v_j$ to $v_i$
and $s_{ij}=0$ otherwise. Obviously, the skew-adjacency matrix is
skew symmetric. The \emph{skew-rank} of an oriented graph
$G^\sigma$, denoted by $sr(G^\sigma)$, is defined as the rank of the
skew-adjacency matrix $S(G^\sigma)$. The \emph{skew-spectrum}
$Sp(G^\sigma)$ of $G^\sigma$ is defined as the spectrum of
$S(G^\sigma)$. Note that $Sp(G^\sigma)$ consists of only purely
imaginary eigenvalues and the skew-rank of an oriented graph is
even.

Let $C_k^\sigma=u_1u_2\cdots u_ku_1$ be an even oriented cycle. The
\emph{sign} of the even cycle $C_k^\sigma$, denoted by
$sgn(C_k^\sigma)$, is defined as the sign of $\prod_{i=1}^k
s_{u_{i}u_{i+1}}$ with $u_{k+1}=u_{1}$. An even oriented cycle
$C_k^\sigma$ is called {\it evenly-oriented} ({\it oddly-oriented})
if its sign is positive (negative). If every even cycle in
$G^\sigma$ is evenly-oriented, then $G^\sigma$ is called {\it
evenly-oriented}. An oriented graph is called an \emph{elementary
oriented graph} if such an oriented graph is $K_2$ or a cycle with
even length. An oriented graph $\mathscr{H}$ is called a \emph{basic
oriented graph} if each component of $\mathscr{H}$ is an elementary
oriented graph.

The oriented graph $G^\sigma$ is called {\it multipartite} if its
underlying graph $G$ is {\it multipartite}. An \textit{induced
subgraph} of $G^\sigma$ is an induced subgraph of $G$ and each edge
preserves the original orientation in $G^\sigma$. For an induced
subgraph $H^\sigma$ of $G^\sigma$, let $G^\sigma-H^\sigma$ be the
subgraph obtained from $G_w$ by deleting all vertices of $H_w$ and
all incident edges. For $V^{\prime}\subseteq V(G^\sigma)$,
$G^\sigma-V^{\prime}$ is the subgraph obtained from $G^\sigma$ by
deleting all vertices in $V^{\prime}$ and all incident edges. A
vertex of a graph $G^\sigma$ is called \textit{pendant} if it is
only adjacent to one vertex, and is called \textit{quasi-pendant} if
it is adjacent to a pendant vertex. A set $M$ of edges in $G^\sigma$
is a \emph{matching} if every vertex of $G^\sigma$ is incident with
at most one edge in $M$. It is \emph{perfect matching} if every
vertex of $G^\sigma$ is incident with exactly one edge in $M$. We
denote by $m_{_{G^\sigma}}(i)$ the number of matchings of $G^\sigma$
with $i$ edges and by $\beta(G^\sigma)$ the \emph{matching number}
of $G^\sigma$ (i.e. the number of edges of a maximum matching in
$G^\sigma$). For an oriented graph $G^\sigma$ on at least two
vertices, a vertex $v\in V(G^\sigma)$ is called {\it unsaturated} in
$G_w$ if there exists a maximum matching $M$ of $G^\sigma$ in which
no edge is incident with $v$; otherwise, $v$ is called {\it
saturated} in $G_w$. Denote by $P_n$, $S_n$, $C_n$, $K_n$ a path, a
star, a cycle and a complete
 graph all of which are simple unoriented graphs of order $n$, respectively.
 $K_{n_1,n_2,\cdots,n_r}$ represents a complete $r$-partite unoriented
 graphs. A
graph is called \emph{trivial} if it has one vertex and no edges.

Recently the study of the skew-adjacency matrix of oriented graphs
attracted some attentions. Cavers et. al \cite{Cavers} provided a
paper about the skew-adjacency matrices in which authors considered
the following topics: graphs whose skew-adjacency matrices are all
cospectral; relations between the matching polynomial of a graph and
the characteristic polynomial of its adjacency and skew-adjacency
matrices; skew-spectral radii and an analogue of the
Perron-Frobenius theorem; and the number of skew-adjacency matrices
of a graph with distinct spectra. Anuradha and Balakrihnan
\cite{anuradha} investigated skew spectrum of the Cartesian product
of an oriented graph with a oriented Hypercube. Anuradha et. al
\cite{anuradha2} considered the skew spectrum of special bipartite
graphs and solved a conjecture of Cui and Hou \cite{cui hou}. Hou et
al \cite{hou1} gave an expression of the coefficients of the
characteristic polynomial of the skew-adjacency matrix
$S(G^\sigma)$. As its applications, they present new combinatorial
proofs of some known results. Moreover, some families of oriented
bipartite graphs with $Sp(S(G^\sigma))=iSp(G)$ were given. Gong et
al \cite{gong1} investigated the coefficients of weighted oriented
graphs. In addition they established recurrences for the
characteristic polynomial and deduced a formula for the matching
polynomial of an arbitrary weighted oriented graph. Xu \cite{xu}
established a relation between the spectral radius and the skew
spectral radius. Also some results on the skew-spectral radius of an
oriented graph and its oriented subgraphs were derived. As
applications, a sharp upper bound of the skew-spectral radius of
oriented unicyclic graphs was present. Some authors investigated the
skew-energy of oriented graphs, one can refer to \cite{adiga, chen,
hou2, gong2, li, tian, zhu}.

This paper is organized as follows. In Section 2, we list some
preliminary results. In Section 3, we characterize the connected
oriented graphs which attaining the skew-rank 2 and determine the
oriented graphs with pendant vertex which attaining the skew-rank 4.
As a consequence, we investigate oriented unicyclic graphs, oriented
bicyclic graphs of order $n$ with pendant vertices which attain the
skew-rank 4, respectively. In Section 4, we determine the skew-rank
of unicyclic graphs of order $n$ with fixed girth in terms of
matching number. Moreover we study the minimum value of skew-rank of
the oriented unicyclic graphs of order $n$ with fixed girth and
characterize oriented graphs with the minimum skew-rank. In Section
5, we consider the non-singularity of the skew-adjacency matrices of
oriented unicyclic graphs.

\section{Preliminary Results}
\ \quad The following results are fundamental. Here we omit their
proofs.

\begin{lemma}\label{basic lemma} 
\begin{enumerate}  [(i).]
\item Let $H^\sigma$ be an induced subgraph of $G^\sigma$. Then $sr(H^\sigma)\leq sr(G^\sigma).$
\item Let $G^\sigma=G_1^\sigma\cup G_2^\sigma \cup \cdots \cup
G_t^\sigma$, where $G_1^\sigma$, $G_2^\sigma$, $\cdots$,
$G_t^\sigma$ are connected components of $G^\sigma$. Then
$sr(G^\sigma)=\sum_{i=1}^{t} sr(G_i^\sigma)$.
\item Let $G^\sigma$ be an oriented graph on $n$ vertices. Then
$sr(G^\sigma)=0$ if and only if $G^\sigma$ is a graph without edges
(empty graph).
\end{enumerate}
\end{lemma}

As we know, the oriented tree and its underlying graph have the same
spectrum \cite{hou1,shader}. So the following is immediate from
\cite{cvetkovic}.
\begin{lemma}\label{rank of tree}
Let $T^\sigma$ be an oriented tree with matching number $\beta(T)$.
Then
$$sr(T^\sigma)=r(T)=2\beta(T).$$
\end{lemma}

The next result is an immediate result of Lemma \ref{rank of tree}.
\begin{lemma}\label{rank of path}
Let $P^\sigma_n$ be an oriented path of order $n$. Then
$sr(P^\sigma_n)=\left\{\begin{array}{cc}n-1,&\mbox{$n$ is
odd,}\\n,&\mbox{$n$ is even.}\end{array}\right.$
\end{lemma}

\begin{lemma}\cite{hou1}\cite{shader}\label{path and cycle}
Let $C_n^\sigma$ be an oriented cycle of order $n$. Then
$$sr(C_n^\sigma)=\left\{\begin{array}{cc}
n,&\mbox{$C_n^\sigma$ is oddly-oriented,}\\
n-2,&\mbox{$C_n^\sigma$ is evenly-oriented,}\\
n-1,&\mbox{otherwise.}\end{array}\right.$$
\end{lemma}

\begin{lemma}\label{deleting pendent vertex}
Let $G^\sigma$ be an oriented graph containing a pendant vertex $v$
with the unique neighbor $u$. Then
$sr(G^\sigma)=sr(G^\sigma-u-v)+2$.
\end{lemma}
\begin{proof}
Assume that all vertices in $V(G^\sigma)$ are indexed by
$\{v_1,v_2,\cdots,v_n\}$ with $v_1=v$, $v_2=u$. Then the
skew-adjacency matrix can be expressed as
$$S(G^\sigma)
=\left(\begin{array}{ccccc}0&s_{12}&0&\cdots&0\\s_{21}&0&s_{23}&\cdots&s_{2n}\\0&s_{32}&0&\cdots&s_{3n}\\
\vdots&\vdots&\vdots&\ddots&\vdots\\0&s_{n2}&s_{n3}&\cdots&0\end{array}\right),
$$
where the first two rows and columns are labeled by $v_1$, $v_2$. So
it follows that
\begin{eqnarray*}
sr(G^\sigma)&=&r\left(\begin{array}{ccccc}0&s_{12}&0&\cdots&0\\s_{21}&0&0&\cdots&0\\0&0&0&\cdots&s_{3n}\\
\vdots&\vdots&\vdots&\ddots&\vdots\\0&0&s_{n3}&\cdots&0\end{array}\right)\\
&=&r\left(\begin{array}{cc}0&s_{12}\\s_{21}&0\end{array}\right)+r\left(\begin{array}{ccc}0&\cdots&s_{3n}\\
\vdots&\ddots&\vdots\\s_{n3}&\cdots&0\end{array}\right)\\
&=&r\left(\begin{array}{cc}0&s_{12}\\s_{21}&0\end{array}\right)+sr(G^\sigma-v_1-v_2)\\
&=&2+sr(G^\sigma-u-v).
\end{eqnarray*}
\end{proof}

\begin{remark}
In fact the result also holds for the unoriented graph, one can
refer to Corollary 1 (pp.234) \cite{cvetkovic}.
\end{remark}

 For convenience, the transformation in Lemma \ref{deleting
pendent vertex} is called $\delta-${\it transformation}. The
skew-rank of some graph can be derived by finite steps of
$\delta-$transformation.

Let $w$ be a common neighbor of two nonadjacent vertices $u$, $v$.
The edges among $u$, $v$ and $w$ have the {\it uniform orientations}
if the arcs is from $u$, $v$ to $w$ or from $w$ to $u$, $v$. The
edges among $u$, $v$ and $w$ have the {\it opposite orientations} if
one arc is from $u$ ($v$) to $w$ and the another is from $w$ to $v$
($u$).

Two nonadjacent vertices $u$, $v$ of an oriented graph $G^\sigma$
are called {\it uniform (opposite) twins} if $N(u)=N(v)$ and the
corresponding edges among $u$, $v$ and each neighbor have the
uniform (opposite) orientations.

 \begin{figure}[ht]
\center
\includegraphics [width = 8cm]{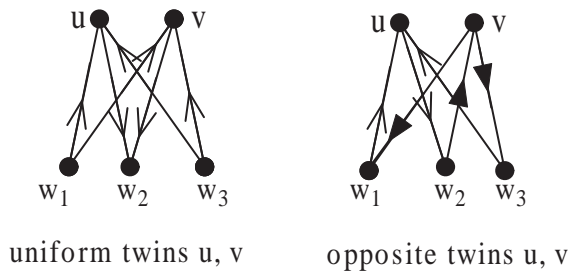}
\caption { \textit{Uniform twins $u,v$ in the left figure, but
opposite twins in the right figure.} }
 \end{figure}

 \begin{example}
Two graphs shown in Fig. 1 contain uniform, opposite twins. $u,v$
are uniform twins in the left graph and opposite twins in the right
graph.
 \end{example}

For an oriented graph $G^\sigma$, the uniform (opposite) twins in
$S(G^\sigma)$ correspond the identical (opposite) rows and columns.
Hence deleting or adding a uniform (opposite) twin vertex does not
change the skew-rank of an oriented graph. Hence we have

\begin{lemma}\label{uniform opposite twins}
Let $u$, $v$ be uniform (opposite) twins of an oriented graph
$G^\sigma$. Then $sr(G^\sigma)=sr(G^\sigma-u)=sr(G^\sigma-v)$.
\end{lemma}

Two pendant vertices are called {\it pendant twins} in $G^\sigma$ if
they have the same neighbor in $G^\sigma$. By Lemma \ref{uniform
opposite twins}, we have
\begin{lemma}
Let $u,v$ be pendant twins of an oriented graph $G^\sigma$. Then
$sr(G^\sigma)=sr(G^\sigma-u)=sr(G^\sigma-v)$.
\end{lemma}

By the definitions of uniform (opposite) twins and evenly-oriented
graph, we can derive the following results.
\begin{lemma}\label{4vertex cycle and twins}
Let $G^\sigma$ be an oriented complete multipartite graph. If all
its $4$-vertex cycles are evenly-oriented, then all vertices in the
same vertex partite set are uniform or opposite twins.
\end{lemma}

\section{Oriented graphs with small skew-rank}

 According to Lemmas \ref{basic lemma} and \ref{rank of path}, it is obvious that $sr(G^\sigma)\geq 2$
 if $G$ is a simple non-empty graph. A natural problem is to characterize the extremal connected oriented graphs
 whose skew-ranks attain the lower bound 2 and the second lower bound
 4.

 \begin{figure}[ht]
\center
\includegraphics [width = 8cm]{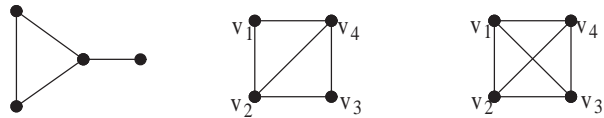}
\caption { \textit{Three graphs $G_1$, $K_{1,1,2}$ and $K_4$ } }
 \end{figure}

 Let $G_1$ be the
graph obtained from $K_3$ by adding a pendant edge to some vertex in
$K_3$ (as depicted in Fig. 2). Let $G^\sigma$ be an oriented graph.
Let $v$ be a vertex of $G^\sigma$ and $V^\prime \subset
V(G^\sigma)$. The notation $N(v)$ represents the neighborhood of $v$
in $G^\sigma$. $G^\sigma[V^\prime]$ denotes the induced subgraph of
$G^\sigma$ on the vertices in $V^\prime$ including the orientations
of edges.

\begin{theorem}\label{three four vertices}
Let $G^\sigma$ be a connected oriented graph of order n $(n=2,3,4)$
with skew-rank 2. Then the following statements hold:
\begin{enumerate}
\item If $n=2$, $G^\sigma$ is an oriented path $P_2^\sigma$ with
arbitrary orientation.
\item If $n=3$, then $G^\sigma$ is $K^\sigma_3$ or $P_3^\sigma$. Each edge has any
orientation in $G^\sigma$.
\item If $n=4$, then $G^\sigma$ is one of the following oriented
graphs with some properties:
\begin{enumerate}
\item Evenly-oriented cycle $C_4^\sigma$.
\item $K^\sigma_{1,3}$ and each
edge has any orientation.
\item Evenly-oriented graph $K_{1,1,2}^\sigma$.
\end{enumerate}
\end{enumerate}
\end{theorem}
\begin{proof}
If $n=2,3$, the results can be easily verified from Lemmas \ref{path
and cycle} and \ref{rank of path}.

If $n=4$, then all $4$-vertex connected unoriented graphs are
$K_{1,3}$, $C_4$, $P_4$, $K_{1,1,2}$, $K_4$, $G_1$ (as depicted in
Fig. 2). By Lemmas \ref{rank of path} and \ref{deleting pendent
vertex} the oriented graphs with $P_4$ or $G_1$ as the underlying
graph have skew-rank 4. And $sr(C_4^\sigma)=4$ if $C_4^\sigma$ is an
oddly-oriented cycle from Lemma \ref{path and cycle}, but the value
is 2 if it is evenly-oriented cycle. If the underlying graph $G$ is
isomorphic to $K_{1,3}$, then $sr(G^\sigma)=2$ and each edge has any
orientation. Next we shall consider the skew-rank of oriented graphs
with $K_{1,1,2}$ or $K_4$ as their underlying graphs.

For convenience, all vertices of $K_{1,1,2}$ are labeled by
$\{v_1,v_2,v_3,v_4\}$ (as depicted in Fig. 2). Then the
skew-adjacency matrix of the oriented graph $K_{1,1,2}^\sigma$ can
be expressed as
$$S(K_{1,1,2}^\sigma)
=\left(\begin{array}{cccc}0&s_{12}&0&s_{14}\\-s_{12}&0&s_{23}&s_{24}\\0&-s_{23}&0&s_{34}\\
-s_{14}&-s_{24}&-s_{34}&0\end{array}\right).$$

Then $$sr(K_{1,1,2}^\sigma)=r\left(\begin{array}{cccc}0&s_{12}&0&0\\-s_{12}&0&0&0\\0&0&0&s_{34}+s_{23}\cdot\frac{s_{14}}{s_{12}}\\
0&0&-s_{34}-s_{23}\cdot\frac{s_{14}}{s_{12}}&0\end{array}\right).$$
So $sr(K_{1,1,2}^\sigma)=2$ if and only if
$s_{34}+s_{23}\cdot\frac{s_{14}}{s_{12}}=0$, i.e.,
$s_{12}s_{34}+s_{14}s_{23}=0$ which implies that the subgraph
$C_4^\sigma$ with vertex set $\{v_1,v_2,v_3,v_4\}$ of
$K_{1,1,2}^\sigma$ is evenly-oriented.

The skew-adjacency matrix of the oriented graph $K_4^\sigma$ can be
expressed as
$$S(K_4^\sigma)
=\left(\begin{array}{cccc}0&s_{12}&s_{13}&s_{14}\\-s_{12}&0&s_{23}&s_{24}\\-s_{13}&-s_{23}&0&s_{34}\\
-s_{14}&-s_{24}&-s_{34}&0\end{array}\right).$$   Then
$$sr(K_4^\sigma)=r\left(\begin{array}{cccc}0&s_{12}&0&0\\-s_{12}&0&0&0\\0&0&0&s_{34}+s_{23}\cdot\frac{s_{14}}{s_{12}}-s_{24}\cdot\frac{s_{13}}{s_{12}}\\
0&0&-s_{34}-s_{23}\cdot\frac{s_{14}}{s_{12}}+s_{24}\cdot\frac{s_{13}}{s_{12}}&0\end{array}\right).$$
Assume that
$s_{34}+s_{23}\cdot\frac{s_{14}}{s_{12}}-s_{24}\cdot\frac{s_{13}}{s_{12}}=0$.
It is equivalent to $s_{12}s_{34}+s_{14}s_{23}=s_{13}s_{24}$.
Obviously the value of the left side is 0, 2 or -2. But the value of
the right side is 1 or -1. So
$s_{34}+s_{23}\cdot\frac{s_{14}}{s_{12}}-s_{24}\cdot\frac{s_{13}}{s_{12}}\neq0$.
Therefore $sr(K_4^\sigma)=4$.
\end{proof}

Next we give a lemma which plays a key role in our proof of Theorem
\ref{skewrank 2}.
\begin{lemma}\label{complete graph} \cite{smith}
A connected graph is not a complete multipartite graph if and only
if it contains $P_4$, $G_1$ (as depicted in Fig. 2) or two copies of
$P_2$ as an induced subgraph.
\end{lemma}

\begin{theorem}\label{skewrank 2}
Let $G^\sigma$ be a connected oriented graph of order $n\geq 5$.
Then $sr(G^\sigma)=2$ if and only if the underlying graph of
$G^\sigma$ is a complete bipartite or tripartite graph and all
$4$-vertex cycles are evenly-oriented in $G^\sigma$.
\end{theorem}

\begin{proof}
{\bf Sufficiency:}

Assume that $G^\sigma$ is a complete bipartite graph $K_{n_1,n_2}$
and all its $4$-vertex cycles are evenly-oriented. Then all vertices
in the same partite vertex set are uniform or opposite twins by
Lemma \ref{4vertex cycle and twins}. Let $X_1,X_2$ be two partite
vertex sets of $K_{n_1,n_2}$. Suppose that $n_1\geq2$. Let $x_1$,
$x_2$ be two arbitrary vertices in $X_1$. By Lemma \ref{uniform
opposite twins}, we have
$sr(K_{n_1,n_2}^\sigma)=sr(K_{n_1,n_2}^\sigma-x_1)=sr(K_{n_1,n_2}^\sigma-x_2)=sr(P_2^\sigma)=2$.

Similarly, $sr(K_{n_1,n_2,n_3}^\sigma)=sr(K_3^\sigma)=2$ if all
$4$-vertex cycles are evenly-oriented in $K_{n_1,n_2,n_3}^\sigma$.

{\bf Necessity:}

Suppose that the underlying graph of $G^\sigma$ is isomorphic to
$K_n$. Since $n\geq5$, $G^\sigma$ must contain $K_4^\sigma$ as an
induced subgraph. So $sr(G^\sigma)\geq sr(K_4^\sigma)=4$ from the
proof of Theorem \ref{three four vertices}. This is a contradiction.

Assume that the underlying graph $G$ is not a complete multipartite
graph. Then $G$ must contain $P_4$, $G_1$ (as depicted in Fig. 2) or
two copies of $P_2$ as an induced subgraph by Lemma \ref{complete
graph}. This implies that $sr(G^\sigma)\geq 4$ which is a
contradiction.

Combining the above discussion, we infer that $G$ is a complete
multipartite graph but not a complete graph. Assume that the
underlying graph $G$ is a complete $t$-partite graph
$K_{n_1,n_2,\cdots,n_t}$. Suppose that $t\geq 4$. Then $G^\sigma$
must contain an induced subgraph $K_4^\sigma$. From the proof of
Theorem \ref{three four vertices}, $sr(G^\sigma)\geq
sr(K_4^\sigma)=4$. So $t=2$ or $3$.

{\bf Case 1.} $t=2$.

Let $X_1$, $X_2$ be the two partite vertex sets of $K_{n_1,n_2}$. If
the cardinality of one of them is one, the $G^\sigma$ is an oriented
star $K_{1,n-1}^\sigma$ and each edge has arbitrary orientation.
Assume that the cardinality of every partite vertex set is more than
one. If $K_{n_1,n_2}^\sigma$ contains an oddly-oriented cycle
$C_4^\sigma$ as an induced subgraph, then
$sr(K_{n_1,n_2}^\sigma)\geq sr(C_4^\sigma)=4$. So all $4$-vertex
cycles in $K_{n_1,n_2}^\sigma$ are evenly-oriented.

{\bf Case 2.} $t=3$.

Similarly to the above discussion, we conclude that all $4$-vertex
cycles in $K_{n_1,n_2,n_3}^\sigma$ are evenly-oriented.
\end{proof}

\begin{theorem}\label{skewrank 4}
Let $G^\sigma$ be an oriented graph with pendant vertex of order
$n$. Then $sr(G^\sigma)=4$ if and only if $G^\sigma$ is one of the
following oriented graphs with some properties:
\begin{enumerate}
\item Graphs obtained by inserting some edges with arbitrary orientation between the center of
$S_{n-n_1-n_2}^\sigma$ $(n_1+n_2\geq 2)$ and some vertices (maybe
partial or all ) of a complete bipartite oriented graph
$K^\sigma_{n_1,n_2}$ such that all $4$-vertex cycles in
$K_{n_1,n_2}^\sigma$ are evenly-oriented.
\item Graphs obtained by inserting some edges with arbitrary orientation between the center of
$S_{n-n_1-n_2-n_3}^\sigma$ $(n_1+n_2+n_3\geq 3)$ and some vertices
(maybe partial or all) of a complete tripartite oriented graph
$K_{n_1,n_2,n_3}^\sigma$ such that all $4$-vertex cycles in
$K_{n_1,n_2,n_3}^\sigma$ are evenly-oriented.
\end{enumerate}
\end{theorem}
\begin{proof}
{\bf Sufficiency:} It is easy to verify that the results hold by
Lemma \ref{deleting pendent vertex} and Theorem \ref{skewrank 2}.

{\bf Necessity:} Assume that $sr(G^\sigma)=4$. Let $x$ be a pendant
vertex in $G^\sigma$ and $N(x)=y$. Suppose that
$G^\sigma-x-y=G^\sigma_{11}\cup G^\sigma_{12}\cup \cdots \cup
G^\sigma_{1t}$ where $G^\sigma_{11}, G^\sigma_{12}, \cdots,
G^\sigma_{1t}$ are connected components of $G^\sigma-x-y$. If each
$G^\sigma_{1i}$ $(i=1,2,\cdots,t)$ is trivial, then $G^\sigma-x-y$
is an oriented star. So $sr(G^\sigma)=2$ which is a contradiction.
Next we shall verify that there exists exactly one nontrivial
connected components in $G^\sigma-x-y$.

Assume that there exist $i$ $(i\geq 2)$ nontrivial connected
components in $G^\sigma-x-y$. Without loss of generality, we denote
them by $G_{11},G_{12},\cdots,G_{1i}$.

{\bf Case 1.} Each of the nontrivial components has no pendant
vertex.

By Lemma \ref{deleting pendent vertex}, we have

\begin{eqnarray*}
sr(G^\sigma)&=&2+sr(G^\sigma-x-y)\\
&=& 2+\sum_{j=1}^isr(G^\sigma_{1j})\\
&\geq&2+\sum_{j=1}^i2\quad\quad \mbox{since $sr(G^\sigma_{1j})\geq 2$}\\
&=&2+2i\geq 6.
\end{eqnarray*}
This is a contradiction.

{\bf Case 2.} There exists one nontrivial component which contains a
pendant vertex.

Let $v$ be the pendant vertex in a nontrivial component and $u$ be
the neighbor of $v$. Then $sr(G^\sigma)=sr(G^\sigma-x-y-u-v)+4$. So
$G^\sigma-x-y-u-v$ is an empty graph. This is impossible since there
exist some edges in other components under our assumption.

Combining the above two cases, there exists exactly one nontrivial
connected component in $G^\sigma-x-y$. Without loss of generality,
assume that $G^\sigma_{11}$ is nontrivial. So
$G^\sigma-x-y=G^\sigma_{11}\cup (n-|G^\sigma_{11}|-2)K_1$. Hence
$sr(G^\sigma)=sr(G^\sigma_{11})+2\geq 4$ with the equality holding
if and only if $sr(G_{11}^\sigma)=2$. So $G_{11}^\sigma$ is one of
the graphs as described in Theorem \ref{skewrank 2}. It is evident
that the subgraph induced by $x,y$ and all isolated vertices in
$G^\sigma-x-y$ is an oriented star $S^\sigma_{n-|G^\sigma_{11}|}$.
Therefore $G^\sigma$ can be obtained by inserting some edges with
any orientation between the center of $S^\sigma_{n-|G^\sigma_{11}|}$
and some vertices (maybe partial or all) of $G^\sigma_{11}$.
\end{proof}

\begin{figure}[ht]
\center
\includegraphics [width = 12cm]{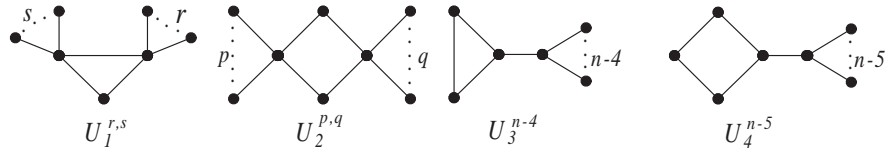}
\caption { \textit{ Four unoriented unicyclic graphs $U_1^{r,s}$,
$U_2^{p,q}$, $U_3^{n-4}$, $U_4^{n-5}$} }
 \end{figure}

By Lemma \ref{path and cycle} and Theorem \ref{skewrank 4}, we have
\begin{theorem}
Let $U^\sigma$ be an oriented unicyclic graph of order $n$ and
$C^\sigma$ be the oriented cycle in $U^\sigma$. Then
$sr(U^\sigma)=4$ if and only if $U^\sigma$ is one of the following
graphs with some properties:
\begin{enumerate}
\item The oddly-oriented cycle $C_4^\sigma$,or the evenly-oriented
cycle $C_6^\sigma$, or the oriented cycle $C_5$ with any
orientation.
\item The oriented graphs with $U_1^{r,s}$ ($r+s=n-3$), $U_2^{p,q}$ ($p+q=n-4$)
or $U_3^{n-4}$ (as depicted in Fig. 3) as the underlying graph and each edge has any
orientation in $U^\sigma$.
\item The oriented graphs with $U_4^{n-5}$ (as depicted in Fig. 3) as the
underlying graph in which $C_4^\sigma$ is an evenly-oriented cycle.
\end{enumerate}
\end{theorem}

\begin{theorem}
Let $B^\sigma$ be an oriented bicyclic graph of order $n$ with
pendant vertices. Then $sr(B^\sigma)=4$ if and only if $B^\sigma$ is
one of the following graphs with some properties:
\begin{enumerate}
\item The oriented graphs with $B_1$, $B_2$ or $B_3$ (as depicted in Fig. 4) as the
underlying graph in which each edge has any orientation.
\item The oriented graphs with $B_4$ or $B_5$ (as depicted in Fig. 4) as the underlying graph in which
the subgraph induced by vertices $u_i$ $(i=1,2,3,4)$ is an
even-oriented cycle.
\item The oriented graphs with $B_6$ or $B_7$ (as depicted in Fig. 4) as the underlying graph such that all $4$-vertex cycles induced by
four vertices among $w_i$ $(i=1,2)$ and $v_j$ $(j=1,2,3)$ are
evenly-oriented.
\end{enumerate}
\end{theorem}

\begin{figure}[ht]
\center
\includegraphics [width = 12cm]{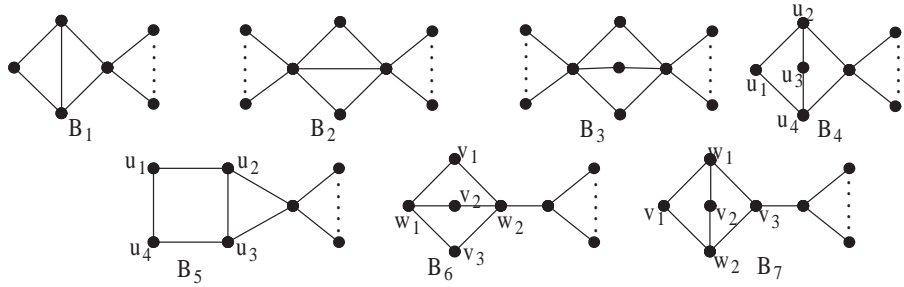}
\caption { \textit{ Seven unoriented bicyclic graphs $B_i$'s
$(i=1,2,\cdots,7)$} }
 \end{figure}

\section{Skew-rank of oriented unicyclic graphs}

In this section we determine the skew-rank of the oriented unicyclic
graphs of order $n$ with girth $k$ in terms of matching number.
Moreover, we investigate the minimum value of the skew-rank among
oriented unicyclic graphs of order $n$ with girth $k$ and
characterize the extremal oriented unicyclic graphs.

\begin{lemma}\label{basic subgraph}\cite{hou1, gong1}
Let $G^\sigma$ be an oriented graph of order $n$ with skew adjacency
matrix $S(G^\sigma)$ and its characteristic polynomial
$$\phi(G^\sigma,\lambda)=\sum_{i=0}^n(-1)^ia_i\lambda^{n-i}=\lambda^n-a_1\lambda^{n-1}+a_2\lambda^{n-2}+\cdots+(-1)^{n-1}a_{n-1}\lambda+(-1)^na_n.$$
Then $$a_i=\sum_{\mathscr{H}}(-1)^{c^+}2^{c}$$ if $i$ is even, where
the summation is over all basic oriented subgraphs $\mathscr{H}$ of
$G^\sigma$ having $i$ vertices and $c^+$, $c$ are the numbers of
evenly-oriented even cycles and even cycles contained in
$\mathscr{H}$, respectively. In particular, $a_i=0$ if $i$ is odd.
\end{lemma}
\begin{theorem}
Let $G^\sigma$ be an oriented unicyclic graph of order $n$ with
girth $k$ and matching number $\beta(G^\sigma)$. Then
$$sr(G^\sigma)=\left\{\begin{array}{cc}2\beta(G^\sigma)-2,&\mbox{if
$C_k$ is evenly-oriented and
$\beta(G^\sigma)=2\beta(G^\sigma-C^\sigma_k)$,}\\\beta(G^\sigma),&\mbox{ortherwise.}\end{array}\right.$$
\end{theorem}
\begin{proof}
If $i>\beta(G^\sigma)$, $G^\sigma$ contains no basic oriented
subgraphs with $2i$ vertices and $a_{2i}=0$. Suppose that $i\leq
\beta(G^\sigma)$. Note that $\lambda^{n-2\beta(G^\sigma)}$ is a
factor of the characteristic polynomial $\phi(G^\sigma,\lambda)$ of
$S(G^\sigma)$, which implies $sr(G^\sigma)\leq 2\beta(G^\sigma)$. So
we consider the coefficient $a_{2\beta(G^\sigma)}$. Next we divide
into three cases to verify this result.

\textbf{Case 1.} $k$ is odd.

Note that there does not exist even cycle in every basic oriented
subgraph $\mathscr{H}$. So
$a_{2\beta(G^\sigma)}=\sum_{\mathscr{H}}(-1)^{0}2^0=\sum_{\mathscr{H}}1\neq0$.
It yields $sr(G^\sigma)=2\beta(G^\sigma)$.

\textbf{Case 2.} $k$ is even and $C_k^\sigma$ is oddly-oriented.

There exists an even cycle in some basic oriented subgraph, but no
evenly-oriented cycle in any basic oriented subgraph. So
$a_{2\beta(G^\sigma)}\neq0$ which implies
$sr(G^\sigma)=2\beta(G^\sigma)$.

\textbf{Case 3.} $k$ is even and $C_k^\sigma$ is evenly-oriented.

Let $\mathcal {H}$ be the set of basic oriented subgraphs on
$2\beta(G^\sigma)$ vertices. Let $\mathcal {H}_1$ be the set of
basic oriented subgraphs on $2\beta(G^\sigma)$ vertices which
contain only $\beta(G^\sigma)$ copies of $K_2$. Let $\mathcal {H}_2$
be the set of basic oriented subgraphs on $2\beta(G^\sigma)$
vertices which contain $C_k^\sigma$ and
$\beta(G^\sigma)-\frac{k}{2}$ copies of $K_2$. Obviously, $\mathcal
{H}=\mathcal {H}_1+\mathcal {H}_2$. Thus
\begin{eqnarray*}
a_{2\beta(G^\sigma)}&=&\sum_{\mathscr{H}\in \mathcal
{H}_1}(-1)^0\cdot 2^0+\sum_{\mathscr{H}\in \mathcal
{H}_2}(-1)^1\cdot
2^1\\
&=&\beta(G^\sigma)-2\beta(G^\sigma-C^\sigma_k).
\end{eqnarray*}
It is evident that $sr(G^\sigma)=2\beta(G^\sigma)$ if
$\beta(G^\sigma)-2\beta(G^\sigma-C^\sigma_k)\neq0$ and
$sr(G^\sigma)<2\beta(G^\sigma)$ if
$\beta(G^\sigma)-2\beta(G^\sigma-C^\sigma_k)=0$. In what follows we
shall verify $sr(G^\sigma)=2\beta(G^\sigma)-2$, i.e.
$a_{2\beta(G^\sigma)-2}\neq 0$, if
$\beta(G^\sigma)-2\beta(G^\sigma-C^\sigma_k)=0$. Let $\mathcal
{H}^\prime_1$ be the set of basic oriented subgraphs on
$2\beta(G^\sigma)-2$ vertices which contain only $\beta(G^\sigma)-1$
copies of $K_2$. Let $\mathcal {H}^\prime_2$ be the set of basic
oriented subgraphs on $2\beta(G^\sigma)-2$ vertices which contain
$C_k^\sigma$ and $\beta(G^\sigma)-\frac{k}{2}-1$ copies of $K_2$. By
Lemma \ref{basic subgraph}, we have
\begin{eqnarray*}
a_{2\beta(G^\sigma)-2}&=&\sum_{\mathscr{H}\in \mathcal
{H}^\prime_1}(-1)^0\cdot
2^0+\sum_{\mathscr{H}\in \mathcal {H}^\prime_2}(-1)^1\cdot 2^1\\
&=&m_{_{G^\sigma}}\big(\beta(G^\sigma)-1\big)-2m_{_{G^\sigma-C_k^\sigma}}\big(\beta(G^\sigma-C^\sigma_k)-1\big).
\end{eqnarray*}
For convenience, we introduce three notations.

$\mathcal {S}_1:\mbox{the set of $(\beta(G^\sigma)-1)$-matchings of
$G^\sigma$};$

$\mathcal {S}_2:\mbox{the set of
$(\beta(G^\sigma-C_k^\sigma)-1)$-matchings of
$G^\sigma-C_k^\sigma$};$

$\mathcal {S}_3=\{M^\prime\mid M^\prime=C_k^\sigma\cup M, \,\, M\in
\mathcal {S}_2\}.$

It is evident that $|\mathcal {S}_1|\geq2|\mathcal {S}_2|$ and
$|\mathcal {S}_2|=|\mathcal {S}_3|$. Next we shall verify that
$m_{_{G^\sigma}}\big(\beta(G^\sigma)-1\big)-2m_{_{G^\sigma-C_k^\sigma}}\big(\beta(G^\sigma-C^\sigma_k)-1\big)\neq
0$. Since $|\mathcal {S}_1|=m_{_{G^\sigma}}(\beta(G^\sigma)-1)$ and
$|\mathcal
{S}_2|=m_{_{G^\sigma-C_k^\sigma}}\big(\beta(G^\sigma-C^\sigma_k)-1\big)$,
so we only verify that $|\mathcal {S}_1|>2|\mathcal {S}_2|$. Note
that $C_k^\sigma$ has exactly two perfect matchings $M_1$, $M_2$
with $\frac{k}{2}$ edges. Suppose that $\mathcal {S}^*=\{M_1\cup M|
M\in \mathcal {S}_2\}\cup \{M_2\cup M| M\in \mathcal {S}_2\}$. So
$|\mathcal {S}^*|=2|\mathcal {S}_2|=2|\mathcal {S}_3|$ and
$|\mathcal {S}^*|\leq |\mathcal {S}_1|$. It is evident that there
exists a $(\beta(G^\sigma)-1)$-matching $M^*$, which is the union of
a matching of $G^\sigma-C_k^\sigma$ with
$\beta(G^\sigma)-\frac{k}{2}$ edges and a matching of $C_k^\sigma$
with $\frac{k}{2}-1$ edges, such that $M^*\in \mathcal {S}_1$ and
$M^*\notin \mathcal {S}^*$. It follows that $|\mathcal {S}_1|\geq
|\mathcal {S}^*|+1=2|\mathcal {S}_2|+1>2|\mathcal {S}_2|$. Thus the
result follows.
\end{proof}

Let $H_{n,k}$ be an underlying graph obtained from $C_k$ by
attaching $n-k$ pendant edges to some vertex on $C_k$.
\begin{theorem}\label{minimum value}
Let $G^\sigma$ be an oriented unicyclic graph of order $n$ with
girth $k$ $(n>k)$. Then
$$sr(G^\sigma)\geq \left\{\begin{array}{cc}k,&\mbox{$k$ is even,}\\k+1,&\mbox{$k$ is odd.}\end{array}\right.$$
This bound is sharp.
\end{theorem}
\begin{proof}
Since $G^\sigma$ must contain $H_{k+1,k}^\sigma$ as an induced
subgraph, so $sr(H_{k+1,k}^\sigma)\leq sr(G^\sigma)$ by Lemma
\ref{basic lemma}. By Lemmas \ref{rank of path} and \ref{deleting
pendent vertex}, we have
$$sr(H_{k+1,k}^\sigma)= \left\{\begin{array}{cc}k,&\mbox{$k$ is
even,}\\k+1,&\mbox{$k$ is odd.}\end{array}\right.$$ Note that all
oriented graphs with $H_{n,k}$ as the underlying graph have the same
skew rank as $H_{k+1,k}^\sigma$. So the result holds.
\end{proof}

The following results can be derived by similar method in Theorems
3.1 and 3.3 in \cite{fan1}.
\begin{lemma}
Let $T^\sigma$ be an oriented tree with $u\in V(T^\sigma)$ and
$G^\sigma_0$ be an oriented graph different from $T^\sigma$. Let
$G^\sigma$ be a graph obtained from $G^\sigma_0$ and $T^\sigma$ by
joining $u$ with certain vertices of $G^\sigma_0$. Then the
following statements hold:
\begin{enumerate}
\item If $u$ is saturated in $T^\sigma$, then
$$sr(G^\sigma)=sr(G^\sigma_0)+sr(T^\sigma).$$
\item If $u$ is unsaturated in $T^\sigma$, then
$$sr(G^\sigma)=sr(T^\sigma-u)+sr(G^\sigma_0+u),$$ where $G^\sigma_0+u$ is
the subgraph of $G^\sigma$ induced by the vertices of $G^\sigma_0$
and $u$.
\end{enumerate}
\end{lemma}

By the above result, we have
\begin{theorem}\label{inertia of unicyclic graphs}
Let $G^\sigma$ be an oriented unicyclic graph and $C^\sigma$ be the
unique oriented cycle in $G^\sigma$. Then the following statements
hold:
\begin{enumerate}
\item If there exists a vertex $v\in V(C^\sigma)$ which is saturated
in $G^\sigma\{v\}$, then
$$sr(G^\sigma)=sr(G^\sigma\{v\})+sr(G^\sigma-G^\sigma\{v\}),$$ where
$G^\sigma\{v\}$ is an oriented tree rooted at $v$ and containing
$v$.
\item If there does not exit a vertex $v\in V(C^\sigma)$ which is saturated in
$G^\sigma\{v\}$, then
$$sr(G^\sigma)=sr(C^\sigma)+sr(G^\sigma-C^\sigma).$$
\end{enumerate}
\end{theorem}

Let $U^*$ be an underlying graph which is obtained from a cycle
$C_k$ and a star $S_{n-k}$ by inserting an edge between a vertex on
$C_k$ and the center of $S_{n-k}$.
\begin{theorem}
Let $G^\sigma$ be an oriented unicyclic graph of order $n$ and
$C_k^\sigma$ be the unique oriented cycle in $G^\sigma$. Assume that
$sr(G^\sigma)=\left\{\begin{array}{cc}k,&\mbox{$k$ is
even,}\\k+1,&\mbox{$k$ is odd.}\end{array}\right.$ Then the
following statements hold:
\begin{enumerate}
\item If there exists a vertex $v\in V(C_k^\sigma)$ which is
saturated in $G^\sigma\{v\}$, then $G^\sigma\{v\}$ is an oriented
star,
$m(G^\sigma-G\{v\})=\left\{\begin{array}{cc}\frac{k-2}{2},&\mbox{$k$
is even,}\\\frac{k-1}{2},&\mbox{$k$ is odd.}\end{array}\right.$ and
$G^\sigma$ has any orientation;
\item If there does not exist a vertex $v\in V(C_k^\sigma)$ which is
saturated in $G^\sigma\{v\}$, then
\begin{enumerate}
\item If $k$ is odd, then $G\cong U^*$ and
$G^\sigma$ has any orientation;
\item If $k$ is even, then $G\cong U^*$ and
$C_k^\sigma$ is evenly-oriented.
\end{enumerate}
\end{enumerate}
\end{theorem}
\begin{proof}
Assume that there exists a vertex $v\in V(C_k^\sigma)$ which is
saturated in $G^\sigma\{v\}$. Note that $G^\sigma\{v\}$ and
$G^\sigma-G^\sigma\{v\}$ are two trees. If $k$ is even, by Lemmas
\ref{rank of tree} and \ref{inertia of unicyclic graphs} we have
\begin{eqnarray*}
sr(G^\sigma)&=&sr(G^\sigma\{v\})+sr(G^\sigma-G^\sigma\{v\})\\
&=&2\beta(G^\sigma\{v\})+2\beta(G^\sigma-G^\sigma\{v\})=k
\end{eqnarray*}

Since $\beta(G^\sigma\{v\})\geq 1$,
$\beta(G^\sigma-G^\sigma\{v\})\geq \frac{k-2}{2}$, so
$\beta(G^\sigma\{v\})= 1$ and $\beta(G^\sigma-G^\sigma\{v\})=
\frac{k-2}{2}$, which implies $G^\sigma\{v\}$ is an oriented star.
From the above process, we can find that this result is independent
of the orientations of edges. So $G^\sigma$ has any orientation.

Similarly the result holds for the case that $k$ is odd.

Suppose that there does not exist a vertex $v\in V(C_k^\sigma)$
which is saturated in $G^\sigma\{v\}$. By Theorem \ref{inertia of
unicyclic graphs}, we have
$$sr(G^\sigma)=sr(C_k^\sigma)+2\beta(G^\sigma-C_k^\sigma).$$

Next we deal with the following three cases.

\textbf{Case 1.} $k$ is odd.

By Lemma \ref{path and cycle} and the above equality, we have
$k+1=k-1+2\beta(G^\sigma-C_k^\sigma)$. It follows that
$\beta(G^\sigma-C_k^\sigma)=1$, i.e. $G^\sigma-C_k^\sigma$ is a
star, and $G^\sigma$ has any orientation.

\textbf{Case 2.} $k$ is even and $C_k^\sigma$ is oddly-oriented.

By the discussion in Case 1, we have $\beta(G^\sigma-C_k^\sigma)=0$.
This contradicts to the fact that there does not exist a vertex
$v\in V(C_k^\sigma)$ which is saturated in $G^\sigma\{v\}$. So this
case can not happen.

\textbf{Case 3.} $k$ is even and $C_k^\sigma$ is evenly-oriented.

By the above discussion, we have $\beta(G^\sigma-C_k^\sigma)=1$,
i.e. $G^\sigma-C_k^\sigma$ is an oriented star.
\end{proof}

\section{Non-singularity of skew-adjacency matrices of oriented
unicyclic graphs}

Let $\mathscr{U}_{n,k}$ be the set of oriented unicyclic graphs of
order $n$ with girth $k$. Let $\mathscr{U}_1$ be the set of oriented
unicyclic graphs of order $n$ with girth $k$ which can be changed to
be an empty (null) graph by finite steps of $\delta$-transformation.
Let $\mathscr{U}_2$ be the set of oriented unicyclic graphs of order
$n$ with girth $k$ which can be changed to be an oriented cycle
$C_k^\sigma$ or the union of isolated vertices and $C_k^\sigma$ by
finite steps of $\delta$-transformation. Obviously,
$\mathscr{U}_{n,k}=\mathscr{U}_1\cup \mathscr{U}_2$.

\begin{theorem}\label{maximum value}
Let $G^\sigma$ be an oriented unicyclic graph of order $n$ with
girth $k$ $(k<n)$. Then
\begin{enumerate}
\item If $G^\sigma\in \mathscr{U}_1$, then $sr(G^\sigma)\leq \left\{\begin{array}{cc}n,&\mbox{$n$ is even,}\\n-1,&\mbox{$n$ is odd.}\end{array}\right.$
\item If $G^\sigma\in \mathscr{U}_2$, then $sr(G^\sigma)\leq
\left\{\begin{array}{cc}n-1,&\mbox{$n$ is odd, $k$ is odd,}\\
n-2,&\mbox{$n$ is even, $k$ is odd,}\\
n,&\mbox{$n$ is even and $C_k^\sigma$ is oddly-oriented,}\\
n-1,&\mbox{$n$ is odd and $C_k^\sigma$ is oddly-oriented,}\\
n-2,&\mbox{$n$ is even and $C_k^\sigma$ is evenly-oriented,}\\
n-3,&\mbox{$n$ is odd and $C_k^\sigma$ is
evenly-oriented.}\end{array}\right.$
\end{enumerate}
\end{theorem}

\begin{proof}
If $G^\sigma\in \mathscr{U}_1$, then by at most
$\lfloor\frac{n}{2}\rfloor$ steps of $\delta-$transformation
$G^\sigma$ can be changed to an empty (null) graph. By Lemma
\ref{deleting pendent vertex}, $sr(G^\sigma)\leq 2\cdot
\lfloor\frac{n}{2}\rfloor$.

If $G^\sigma\in \mathscr{U}_2$, then by at most
$\lfloor\frac{n-k}{2}\rfloor$ steps of $\delta-$transformation
$G^\sigma$ can be changed to be oriented cycle $C_k^\sigma$ or the
union of isolated vertices and $C_k^\sigma$. By Lemma \ref{deleting
pendent vertex}, $sr(G^\sigma)\leq 2\cdot
\lfloor\frac{n-k}{2}\rfloor+sr(C_k^\sigma)$. The result holds by
Lemma \ref{path and cycle}.
\end{proof}

In what follows we consider the non-singularity of skew-adjacency
matrices of oriented unicyclic graphs. As we know, if the order $n$
is odd, then the oriented unicyclic graph must be singular. So we
only need consider the oriented unicyclic graph with even order. By
Theorem \ref{maximum value}, we have
\begin{theorem}
Let $G^\sigma$ be an oriented unicyclic graph with even order $n$.
Then $S(G^\sigma)$ is nonsingular if and only if $G^\sigma\in
\mathscr{U}_1$ and $G^\sigma$ has a perfect matching, or
$G^\sigma\in \mathscr{U}_2$, $C_k^\sigma$ is oddly-oriented and
$G^\sigma-C_k^\sigma$ has a perfect matching.
\end{theorem}

\end{document}